\documentclass[11pt,a4paper]{article}

\usepackage{amsmath,amssymb,amsthm,mathrsfs}
\usepackage[margin=3.2cm]{geometry}
\usepackage{tikz}
\usepackage[explicit]{titlesec}
\titleformat{\section}[block]{\normalfont\bfseries\filcenter}{\thesection.}{0.6em}{\MakeUppercase{#1}}
\titleformat{\subsection}[block]{\normalfont\bfseries\filcenter}{\thesubsection.}{0.6em}{#1}
\usepackage{hyperref}

\newtheorem{theorem}{Theorem}[section]
\newtheorem{proposition}[theorem]{Proposition}

\newtheorem{corollary}[theorem]{Corollary}
\theoremstyle{definition}
\newtheorem{definition}[theorem]{Definition}
\newtheorem{example}[theorem]{Example}
\newtheorem{question}[theorem]{Question}

\theoremstyle{remark}
\newtheorem{remark}[theorem]{Remark}

\newcommand{\colim}{\operatorname*{colim}}
\newcommand{\Cat}{\mathcal{C}}
\newcommand{\lowlim}{\underline{\lim}}
\newcommand{\uplim}{\overline{\lim}}
\newcommand{\Vect}{\mathbf{Vect}}
\newcommand{\Set}{\mathbf{Set}}
\newcommand{\LCS}{\mathbf{LCS}}
\newcommand{\Ban}{\mathbf{Ban}}

\title{Upper and lower limits of scales\\ and the interchange of quantifiers}
\author{%
  Zohreh Eskandarian\thanks{Kazan Federal University, Kremlievskaia ul. 18,
    Kazan, Tatarstan, 420008, Russia. E-mail:
    \texttt{zohreh.eskandarian@gmail.com}}
  \and
  Zamira Kholmatova\thanks{Innopolis University, Universitetskaia ul. 1,
    Innopolis, Tatarstan, 420500, Russia. E-mail:
    \texttt{zamira.kholmatova@gmail.com}}
  \and
  Andrej Novikov\thanks{Novikov Laboratories LLC, Kazan, Tatarstan, Russia.
    Corresponding author. E-mail:
    \href{mailto:a.hobukob@gmail.com}{\texttt{a.hobukob@gmail.com}}}%
}
\date{}

\begin{document}
\maketitle

\begin{abstract}
We define lower and upper limits for a family of structured subobjects of a
fixed ambient object, as the direct limit of the meets over tails and the
inverse limit of the joins over tails, and construct the canonical comparison
morphism between them. The family need not be a diagram: its members are
related only through the ambient object. We distinguish two senses in which
the comparison may be invertible, before and after forgetting the structure,
and prove that they coincide exactly when the forgetful functor reflects
isomorphisms; they therefore agree for sets, for vector spaces and for Banach
spaces, and separate for locally convex spaces. For countable scales of Banach spaces compatibly embedded in a common
Hausdorff locally convex space, the separation cannot occur when the upper
limit is ultrabornological, so under those hypotheses the obstruction is localised in the topology of
that limit. Intermediate
conditions correspond to factorisations of the forgetful functor, of which
bornological convergence is one. In the cases that matter the comparison is an
interchange of quantifiers: membership in the lower limit means that one value
of the second index serves every value of the first, and membership in the
upper limit that every value of the first is served by some value of the
second. We exhibit this for spaces of test functions, for weighted Sobolev
spaces, for Laurent series and for the adeles, and discuss three further
conditions that suggest themselves.


\end{abstract}

\section{Introduction}

For subsets $A_n$ of a set, $\liminf_nA_n=\bigcup_k\bigcap_{n\geq k}A_n$ and
$\limsup_nA_n=\bigcap_k\bigcup_{n\geq k}A_n$, and the first is contained in
the second. Intersections and unions are limits and colimits, so the two
constructions iterate a limit and a colimit in the two orders, and the
inclusion between them is a comparison morphism. This paper carries the
construction over to families of structured objects and asks when the
comparison morphism is invertible.

Once the objects carry structure the question has two answers rather than one.
The comparison may be an isomorphism, which for locally convex spaces is the
interchangeability of the projective and the inductive limit; or it may be one
only after the structure is forgotten, which is the assertion that the two
limits have the same elements. We call these (C1) and (C2). They differ
exactly when the forgetful functor fails to reflect isomorphisms
(Proposition~\ref{prop:reflect}), so they agree for sets, for vector spaces
and for Banach spaces, by the open mapping theorem, and can differ for locally
convex spaces. Under the hypotheses of Proposition~\ref{prop:ultrab} the obstruction is
localised in the topology of the upper limit, and disappears when that limit
is ultrabornological.

Two things make the distinctions vanish (Section~\ref{sec:degen}). A family
that ascends or descends has both limits equal to an ordinary direct or
inverse limit, with the comparison the identity; so the construction carries
information only for families that change direction infinitely often. And the two conditions coincide whenever the forgetful functor reflects
isomorphisms, which it does for sets, vector spaces and Banach spaces, so the
distinction between them is a phenomenon of locally convex spaces.

The construction is not an end in itself. In the cases that matter the
comparison is an interchange of quantifiers: the lower limit is described by a
condition $\exists m\,\forall n$ and the upper by $\forall n\,\exists m$.
Section~\ref{sec:quant} shows this for spaces of test functions, for weighted
Sobolev spaces, for Laurent series, where the two limits agree, and for the
adeles, where the lower limit is $\mathbb{A}_f$,
the condition ``integral at almost all places'' being exactly
$\exists S\,\forall v$, and where the failure of the two limits to agree is no
defect: the lower one is the object on which arithmetic is done.

The construction arose in the authors' study of scales of spaces attached to
the powers of a positive operator affiliated with a von Neumann algebra
\cite{Nov2017,NE2016,Esk2018}, where the same interchange governs the mixed
corners left by a non-central spectral projection. That application requires a
separate discussion and will be given elsewhere.

The comparison morphism and the question of which index shapes commute for all
diagrams are classical \cite{MacLane,Foltz,BJLS}, as is the interchange of
projective and inductive limits for weighted spaces \cite{ABB,Wengenroth}. New here are the treatment of families rather than diagrams, the description
of the two senses of convergence and of when they differ, the recognition of
the interchange of quantifiers across the examples of
Section~\ref{sec:quant}, and the discussion of the conditions we do not
adopt in Section~\ref{sec:other}.

\subsection*{Main contributions}

\begin{itemize}
\item The notion of a scale over an ambient object, and the construction of
lower and upper limits and of the comparison morphism for families that are
not diagrams (Definitions~\ref{def:scale}, \ref{def:limits},
Proposition~\ref{prop:comparison}).
\item The criterion separating the two senses of convergence:
they coincide for all scales precisely when the forgetful functor reflects
isomorphisms (Proposition~\ref{prop:reflect}), with an explicit locally convex
counterexample (Theorem~\ref{thm:separate}).
\item The localisation of the obstruction: for countable scales of Banach
spaces compatibly embedded in a common Hausdorff locally convex space,
algebraic convergence implies structural convergence as soon as the upper
limit is ultrabornological (Proposition~\ref{prop:ultrab}).
\item The description of intermediate conditions: they are exactly what
factorisations of the forgetful functor produce, and are genuine precisely
when structure is lost at that stage (Proposition~\ref{prop:factor}), the
bornological one being an instance (Corollary~\ref{prop:born}).
\item The reading of the comparison as an interchange of quantifiers, and its
realisation in four classical settings (Section~\ref{sec:quant}).
\end{itemize}

\section{Preliminaries}\label{sec:prelim}

\subsection{Conventions}

$\mathbb{K}$ denotes $\mathbb{R}$ or $\mathbb{C}$. $\Set$, $\Vect$, $\Ban$ and
$\LCS$ denote sets, vector spaces over $\mathbb{K}$, Banach spaces
and Hausdorff locally convex spaces, morphisms in the last two being
continuous linear maps. A category is \emph{thin} if it has at most one
morphism between any two objects. $\LCS$ is not balanced: a continuous linear
bijection need not be open. For a directed poset $N$ and $k\in N$ we write
$N_k=\{n\in N:n\geq k\}$ for the tail at $k$. For subobjects of a fixed object
we write $\bigwedge$ and $\bigvee$ for meet and join; in $\LCS$ these are the
intersection with the initial structure and the sum with the final structure.

\subsection{Limits and colimits}

If $\Cat$ has $K^{\mathrm{op}}$-limits and $N$-colimits and
$F:K^{\mathrm{op}}\times N\to\Cat$, there is a
canonical morphism $\colim_N\lim_{K^{\mathrm{op}}}F\to\lim_{K^{\mathrm{op}}}\colim_NF$ \cite[IX.2]{MacLane}.
Filtered colimits commute with finite limits in $\Set$ and $\Vect$; which pairs
of shapes commute for all diagrams is studied in \cite{Foltz,BJLS}. The
question here is the complementary one, for individual families. A countable
inductive limit of Fréchet spaces is an $(LF)$-space and a countable
projective limit of $(LB)$-spaces a $(PLB)$-space; for weighted spaces of
continuous functions the coincidence of the two is characterised by a
condition~(D) on the defining weights \cite{ABB}, and the derived projective
limit functor is treated in \cite{Wengenroth}.

\section{Scales and the comparison morphism}\label{sec:scales}

The input of the construction is a family and not a diagram. A family
alternating between two subsets of a set has a $\liminf$ and a $\limsup$
without being a functor on its index set; what its members share is not
transition maps but a common ambient object.

\begin{definition}\label{def:scale}
Let $U:\Cat\to\Cat_0$ be a faithful functor, $V$ an object of $\Cat_0$, and
$N$ a directed poset. A \emph{scale} over $V$ indexed by $N$ is a family
$D=(D_n)_{n\in N}$ of objects of $\Cat$ together with monomorphisms
$\iota_n:UD_n\to V$ in $\Cat_0$.
\end{definition}

We assume throughout:
\begin{itemize}
\item[(H1)] for each $k$ the meet $C_k=\bigwedge_{n\in N_k}D_n$ and the join
$E_k=\bigvee_{n\in N_k}D_n$ exist in $\Cat$, taken over the subobjects
$\iota_n(UD_n)\subseteq V$;
\item[(H2)] $(C_k)_{k\in N}$ is a directed system and $(E_k)_{k\in N}$ an
inverse system in $\Cat$, and $\colim_kC_k$ and $\lim_kE_k$ exist;
\item[(H3)] every morphism occurring between the subobjects
$\iota_n(UD_n)$, $UC_k$ and $UE_k$ of $V$ has as underlying map in $\Cat_0$
the inclusion of the corresponding subobject of $V$.
\end{itemize}

Condition (H3) is what makes the construction possible at all, and it is not
implied by faithfulness of $U$. A scale is a family and not a diagram, so
there are no morphisms $D_n\to D_{n'}$ to appeal to; what replaces them is
that all the maps in sight are inclusions inside $V$. In the intended
examples, $\Cat=\LCS$ and $\Cat_0=\Vect$ with the $D_n$ literally subspaces of
$V$, (H3) holds by construction.

In $\LCS$ over $\Vect$ the meet of a family of subspaces of $V$ is their
intersection with the initial locally convex structure and the join is their
sum with the final one; in particular, the meet carries the supremum of the structures and the join their infimum, so
that if all members have the same underlying space and their structures are
comparable, the meet carries the finest and the join the coarsest.

A join in $\LCS$ need not be Hausdorff; this is covered by (H1), which asks
that the objects exist in $\Cat$.

\begin{definition}\label{def:limits}
The \emph{lower} and \emph{upper limits} of a scale $D$ are
\[
  \lowlim D=\colim_{k\in N}\ \bigwedge_{n\in N_k}D_n,
  \qquad
  \uplim D=\lim_{k\in N}\ \bigvee_{n\in N_k}D_n.
\]
\end{definition}

For $\Cat=\Cat_0=\Set$ with $U=\mathrm{id}$, (H1) and (H2) hold automatically
and Definition~\ref{def:limits} returns the classical upper and lower limits.
For $\Cat=\LCS$ and $\Cat_0=\Vect$, with $N$ countable and the $C_k$ Fréchet, $\lowlim D$ is an $(LF)$-space and $\uplim D$ a projective limit of the $E_k$; it is a $(PLB)$-space when the $E_k$ are $(LB)$-spaces, in the sense of \S\ref{sec:prelim}.

\begin{proposition}\label{prop:comparison}
Under \textnormal{(H1)--(H3)} there is a unique morphism
$\Lambda_D:\lowlim D\to\uplim D$ in $\Cat$ such that for all $k\leq m$ in $N$
the composite $C_k\to\lowlim D\xrightarrow{\Lambda_D}\uplim D\to E_m$ equals
the composite $C_k\to D_n\to E_m$ for $n\in N_m$.
\end{proposition}

\begin{proof}
Fix $k\leq m$. For $n\in N_m\subseteq N_k$ the composite $C_k\to D_n\to E_m$
is independent of $n$: by (H3) its underlying map in $\Cat_0$ is the inclusion
$UC_k\subseteq UE_m$ of subobjects of $V$, which does not involve $n$, and $U$
is faithful. This yields morphisms $\varphi_{km}:C_k\to E_m$
compatible with the transition morphisms of both systems, hence a morphism
$C_k\to\lim_mE_m$ for each $k$, and these form a cocone over $(C_k)$, giving a
unique morphism $\colim_kC_k\to\lim_mE_m$. Uniqueness is the universal
property of the colimit.
\end{proof}

\begin{remark}\label{rem:hyp}
Of the hypotheses, (H1) and (H2) are existence assumptions and can be dropped
wherever the objects in question happen to exist. (H3) is not of that kind: a scale carries no morphisms $D_n\to D_{n'}$, so the
independence of $n$ in Proposition~\ref{prop:comparison} has to come from
somewhere, and (H3) is what supplies it. It holds
whenever the $D_n$ are literally subobjects of $V$ and the morphisms between
them are the inclusions, which is the case in every example below. We do not
know a useful weakening of it.

This is also the point at which the present setting differs from that of a
two-variable diagram $F:K^{\mathrm{op}}\times N\to\Cat$, for which the
canonical morphism is classical \cite[IX.2]{MacLane}. A scale is not such a
diagram: its members are related only through the ambient object, and the
tails supply the diagram structure afterwards. Families that alternate, as
most of our examples do, are not functors on their index set at all.
\end{remark}

\subsection{The embedding order}\label{sec:order}

\begin{definition}\label{def:embord}
For subobjects $D,D'$ of $V$ write $D\sqsubseteq D'$ if $UD\subseteq UD'$ in
$\Cat_0$ and the induced map $D\to D'$ is a morphism of $\Cat$.
\end{definition}

The symbol $\sqsubseteq$ records a factorisation of monomorphisms together
with the existence of a morphism in $\Cat$; it is in no sense a numerical
inequality, and we avoid the word \emph{monotone} for that reason. In general it
is only a preorder: $D\sqsubseteq D'\sqsubseteq D$ need not force $D=D'$, as
the normed spaces $(\mathbb{K}^{(\mathbb{N})},\|\cdot\|_1)$ and
$(\mathbb{K}^{(\mathbb{N})},2\|\cdot\|_1)$ show when they are taken as
distinct objects of a category of normed spaces with bounded maps. In $\LCS$
it is a partial order, two comparable locally convex topologies on one space
with both identities continuous being equal.

Call a scale $D$ \emph{ascending} if $D_n\sqsubseteq D_m$ whenever $n\leq m$,
and \emph{descending} if $D_m\sqsubseteq D_n$ whenever $n\leq m$. Neither
condition is implied by the other, and a scale may be neither.

\begin{proposition}\label{prop:envelope}
Let $D$ be a scale satisfying \textnormal{(H1)}. Then $(C_k)$ is ascending
with $C_k\sqsubseteq D_k$, and is the greatest ascending scale with that property;
and $(E_k)$ is descending with $D_k\sqsubseteq E_k$, and is the least descending
scale with that property.
\end{proposition}

\begin{proof}
If $k\leq k'$ then $N_{k'}\subseteq N_k$, so $C_k\sqsubseteq C_{k'}$, and
$k\in N_k$ gives $C_k\sqsubseteq D_k$. If $G$ is ascending with
$G_n\sqsubseteq D_n$, then for $n\in N_k$ we have
$G_k\sqsubseteq G_n\sqsubseteq D_n$, hence
$G_k\sqsubseteq\bigwedge_{n\in N_k}D_n=C_k$. The second statement is dual.
\end{proof}

Thus
$\lowlim D=\colim_N(C_k)$ and $\uplim D=\lim_N(E_k)$ are an inductive limit of
an ascending scale and an inverse limit of a descending one: the two
constructions are the direct and inverse limits applied to the ascending and
descending envelopes of $D$.

The construction may be pictured as follows. The upper row is the ascending
envelope, the lower row the descending one, and the two limits are the direct
limit of the first and the inverse limit of the second.
\[
\begin{tikzpicture}[>=latex, baseline=(D2.base)]
  \node (C1) at (0,1.5)    {$C_1$};
  \node (C2) at (2.2,1.5)  {$C_2$};
  \node (C3) at (4.4,1.5)  {$\cdots$};
  \node (L)  at (7.1,1.5)  {$\lowlim D$};
  \node (D1) at (0,0)      {$D_1$};
  \node (D2) at (2.2,0)    {$D_2$};
  \node (D3) at (4.4,0)    {$\cdots$};
  \node (E1) at (0,-1.5)   {$E_1$};
  \node (E2) at (2.2,-1.5) {$E_2$};
  \node (E3) at (4.4,-1.5) {$\cdots$};
  \node (U)  at (7.1,-1.5) {$\uplim D$};
  \draw[->] (C1)--(C2); \draw[->] (C2)--(C3); \draw[->] (C3)--(L);
  \draw[->] (U)--(E3);  \draw[->] (E3)--(E2); \draw[->] (E2)--(E1);
  \draw[->] (C1)--(D1); \draw[->] (D1)--(E1);
  \draw[->] (C2)--(D2); \draw[->] (D2)--(E2);
  \draw[->] (L)--(U) node[midway,right] {$\Lambda_D$};
  \node at (5.9,1.9)  {\footnotesize $\colim_k$};
  \node at (5.9,-1.1) {\footnotesize $\lim_k$};
\end{tikzpicture}
\]
The vertical arrows on the left are $C_k\sqsubseteq D_k\sqsubseteq E_k$ of
Proposition~\ref{prop:envelope}. The scale $D$ itself carries no arrows: it is
a family and not a diagram, and the horizontal arrows exist only after passage
to the envelopes. In particular the limits are invariant under passage to a
tail but not to an arbitrary cofinal subset: for subsets $A_n=\{0\}$ with $n$
even and $A_n=\{n\}$ with $n$ odd the lower limit is $\emptyset$ over
$\mathbb{N}$ and $\{0\}$ over $2\mathbb{N}$.

\section{Two senses of convergence}\label{sec:grades}

\begin{definition}\label{def:grades}
Let $D$ be a scale over $V$. We say that $D$ is \emph{structurally convergent}, or satisfies
\textnormal{(C1)}, if $\Lambda_D$ is an isomorphism in $\Cat$, and
\emph{algebraically convergent}, or satisfies \textnormal{(C2)}, if
$U\Lambda_D$ is an isomorphism in $\Cat_0$. We use the labels (C1) and (C2)
throughout.
\end{definition}

Both conditions occur in the literature. In the lattice of subsets of a set the two coincide and are
Painlev\'e--Kuratowski convergence for the discrete topology, that is,
set-theoretic convergence. For scales of
locally convex spaces, (C1) is the interchangeability of the projective and
the inductive limit, characterised for weighted spaces of continuous functions
by condition~(D) of \cite{ABB}; the underlying spaces of the two limits may
coincide there while the topologies do not, which is (C2) without (C1).
Condition (C2) is also what is asserted when two limit constructions are said
to have the same elements, and it is all that is needed to transfer questions
of membership, of dimension and of algebraic identities from one limit to the
other.

\subsection{When the two agree}

\begin{proposition}\label{prop:reflect}
\textnormal{(C1)} implies \textnormal{(C2)}. Conversely, if $U$ reflects
isomorphisms, then \textnormal{(C2)} implies \textnormal{(C1)} for every
scale; and if $U$ does not reflect isomorphisms, there is a scale satisfying
\textnormal{(C2)} and not \textnormal{(C1)}.
\end{proposition}

\begin{proof}
The first two assertions are immediate. For the third, let $f:X\to Y$ be a
morphism with $Uf$ an isomorphism and $f$ not; take $N=\mathbb{N}$, $V=UY$, and
the scale equal to $X$ for $n$ even and $Y$ for $n$ odd, with $UX\to V$ given
by $Uf$. Then $C_k=X$ and $E_k=Y$ for every $k$, so $\Lambda_D=f$.
\end{proof}

\noindent The hypothesis holds in the familiar cases. It is trivial when
$U$ is the identity, as for scales of sets, of vector spaces or in a poset. It
holds for $\Cat=\Ban$ and $\Cat_0=\Vect$, by the open mapping theorem. It holds for finite-dimensional objects of $\LCS$ over $\Vect$, a
finite-dimensional vector space carrying only one Hausdorff linear
topology. It fails for $\Cat=\LCS$, since a continuous linear bijection need not be
open. Proposition~\ref{prop:reflect} thus identifies the categorical mechanism
responsible for the distinction between the two conditions.

\begin{theorem}\label{thm:separate}
Let $\Cat=\LCS$ and $\Cat_0=\Vect$. The scale
$D_n=(\mathbb{K}^{(\mathbb{N})},\|\cdot\|_1)$ for $n$ even and
$D_n=(\mathbb{K}^{(\mathbb{N})},\|\cdot\|_2)$ for $n$ odd satisfies
\textnormal{(C2)} and not \textnormal{(C1)}.
\end{theorem}

\begin{proof}
All $D_n$ have the same underlying space, so $C_k$ carries the finer and
$E_k$ the coarser of the two norms; since $\|x\|_2\leq\|x\|_1$ this gives
$C_k=(\mathbb{K}^{(\mathbb{N})},\|\cdot\|_1)$ and
$E_k=(\mathbb{K}^{(\mathbb{N})},\|\cdot\|_2)$, and $\Lambda_D$ is the identity
of $\mathbb{K}^{(\mathbb{N})}$, a continuous bijection. It is not open: for
$x_m=m^{-1}(e_1+\dots+e_m)$ one has $\|x_m\|_1=1$ and
$\|x_m\|_2=m^{-1/2}\to0$.
\end{proof}

\begin{remark}\label{rem:isoweak}
Neither condition may be replaced by the weaker requirement that the two
limits be isomorphic. In $\LCS$ let $V=\ell^2$, $H=\{x\in\ell^2:x_1=0\}$, and
let $D_n=\ell^2$ for $n$ even and $D_n=H$ for $n$ odd. Then $\lowlim D=H$ and
$\uplim D=\ell^2$, so $\Lambda_D$ is a proper inclusion and both conditions
fail, while $H$ and $\ell^2$ are isometrically isomorphic. In a category whose
objects are isomorphic to proper subobjects of themselves the limits may be
isomorphic while differing by an arbitrary codimension, so isomorphism of the
limits carries no information about $\Lambda_D$; it is useful only for transferring
properties that are themselves invariant under isomorphism.
\end{remark}

\subsection{Where the obstruction lies}

\begin{proposition}\label{prop:ultrab}
Let $N$ be countable, let every $D_n$ be a Banach space, and assume that $V$ embeds linearly in a
Hausdorff locally convex space $G$ so that every $D_n\to G$ is continuous.
Then each $C_k=\bigwedge_{n\in N_k}D_n$ is a Fréchet space and $\lowlim D$ is
an $(LF)$-space, hence webbed. If $\uplim D$ is ultrabornological, then
\textnormal{(C2)} implies \textnormal{(C1)}.
\end{proposition}

\begin{proof}
By Proposition~\ref{prop:comparison} the morphism
$\Lambda_D:\lowlim D\to\uplim D$ is continuous, and by (C2) it is bijective.
Fréchet spaces are webbed and the class of webbed spaces is stable under
countable inductive limits, so $\lowlim D$ is webbed. De Wilde's open mapping
theorem \cite[24.30]{MV} applies to a continuous linear surjection from a
webbed space onto an ultrabornological space, whence $\Lambda_D$ is open.
\end{proof}

\begin{corollary}\label{cor:frechet}
Under the hypotheses of Proposition~\ref{prop:ultrab}, if $\uplim D$ is a
Fréchet space then \textnormal{(C2)} implies \textnormal{(C1)}, a Fréchet
space being ultrabornological.
\end{corollary}

\begin{remark}\label{rem:ultrab}
Proposition~\ref{prop:ultrab} says where the obstruction can lie. Under its
hypotheses the lower limit is an $(LF)$-space and therefore webbed, so the
obstruction lies in the upper limit, which is a $(PLB)$-type space and need not be ultrabornological;
when it is, the two conditions coincide. This is consistent with the
separations exhibited here: in Theorem~\ref{thm:separate} the upper limit is a
non-complete normed space, to which the proposition does not apply, and for
the weighted scales of \cite{ABB} it is a $(PLB)$-space whose ultrabornologicity is itself one of the
conditions studied there. The proposition therefore says where the two conditions
cannot separate, and it is the reason that the classical settings, where the
limits are Fréchet or Banach, exhibit a single notion of convergence.
\end{remark}

\subsection{Intermediate conditions}

\begin{proposition}\label{prop:factor}
Let $U=U_2\circ U_1$ with $U_1:\Cat\to\mathcal{D}$ and $U_2:\mathcal{D}\to\Cat_0$
faithful, and say that a scale satisfies \textnormal{(C$_{\mathcal{D}}$)} if
$U_1\Lambda_D$ is an isomorphism. Then
\[
  \textnormal{(C1)}\ \Longrightarrow\ \textnormal{(C$_{\mathcal{D}}$)}
  \ \Longrightarrow\ \textnormal{(C2)},
\]
the first implication is strict if and only if $U_1$ fails to reflect
isomorphisms, and the second is strict if and only if there is a morphism $f$
of $\Cat$ with $Uf$ an isomorphism and $U_1f$ not.
\end{proposition}

\begin{proof}
The implications are immediate. For strictness, argue as in
Proposition~\ref{prop:reflect}: a morphism $f:X\to Y$ of $\Cat$ with $U_1f$ an
isomorphism and $f$ not gives the alternating scale with $\Lambda_D=f$, which
satisfies (C$_{\mathcal{D}}$) and not (C1); and a morphism $f$ of $\Cat$ with $Uf$ an isomorphism and $U_1f$ not gives one
satisfying (C2) and not (C$_{\mathcal{D}}$). The second criterion is stated
for morphisms of $\Cat$ because the construction needs such an $f$; failure of
$U_2$ to reflect isomorphisms on all of $\mathcal{D}$ would not suffice.
\end{proof}

\begin{remark}\label{rem:structure}
Intermediate conditions are thus not invented but read off from intermediate
structure. A factorisation of the forgetful functor produces one, and it is a
genuinely new condition exactly when structure is genuinely lost at that
stage. Where there is nothing to forget, as for scales of sets or of vector
spaces with $U$ the identity, there is nothing between: (C1) and (C2) coincide
and no factorisation can separate them. The richer the ambient structure, the
longer the chain of conditions that can be inserted, and none of them makes
sense without the structure that defines it.
\end{remark}

\begin{corollary}\label{prop:born}
Let $\Cat=\LCS$, let $\mathcal{D}$ be the category of convex bornological
spaces with bounded linear maps \cite{HN}, and let $U_1$ assign to a locally
convex space its von Neumann bornology. The resulting condition, that
$U\Lambda_D$ be bijective with $\Lambda_D$ and its inverse carrying bounded
sets to bounded sets, lies strictly between \textnormal{(C2)} and
\textnormal{(C1)}.
\end{corollary}

\begin{proof}
By Proposition~\ref{prop:factor} it suffices to exhibit the two failures of
reflection. For $U_1$: let $A=(\ell^2,\|\cdot\|)$ and
$B=(\ell^2,\sigma(\ell^2,\ell^2))$ and let $D_n=A$ for $n$ even and $D_n=B$
for $n$ odd. The two topologies live on one space and the weak one is coarser,
so $C_k=A$, $E_k=B$ and $\Lambda_D$ is the identity of $\ell^2$; the weakly
bounded sets are the bounded ones by the uniform boundedness principle, while
the two topologies differ. For $U_2$: in the scale of
Theorem~\ref{thm:separate} the sequence $x_m=m^{-1/2}(e_1+\dots+e_m)$ has
$\|x_m\|_2=1$ and $\|x_m\|_1=\sqrt{m}$, so $\Lambda_D^{-1}$ does not preserve
boundedness.
\end{proof}

\section{Degeneration}\label{sec:degen}

The construction degenerates on ascending and descending families, and independently of that
the two conditions of Section~\ref{sec:grades} coincide when the forgetful
functor reflects isomorphisms. We record both.

\begin{theorem}[Degeneration]\label{thm:degenerate}
Let $D$ satisfy \textnormal{(H1)--(H3)}. If $D$ is descending, then
$C_k=\bigwedge_ND$ and $E_k=D_k$ for every $k$, so both limits equal the
inverse limit $\bigwedge_ND$ of $D$ and $\Lambda_D$ is the identity. If $D$
is ascending, then $C_k=D_k$ and $E_k=\bigvee_ND$ for every $k$, so both
limits equal the direct limit $\bigvee_ND$ of $D$ and $\Lambda_D$ is the
identity.
\end{theorem}

\begin{proof}
Let $D$ be descending and fix $k$. For $n_0\in N$ choose $n_1\geq k,n_0$; then
$C_k\sqsubseteq D_{n_1}\sqsubseteq D_{n_0}$, so $C_k\sqsubseteq\bigwedge_ND$,
the reverse relation being clear. Hence $C_k$ is independent of $k$ and $\lowlim D=\bigwedge_ND$. The subobject
$D_k$ is the greatest among $(D_n)_{n\in N_k}$, so $E_k=D_k$ as subobjects;
the final structure on $E_k$ is the finest making every $D_n\to E_k$ a
morphism, and since the identity of $D_k$ is among these and each
$D_n\to D_k$ is already a morphism for $n\geq k$, it coincides with that of
$D_k$. Hence $\uplim D=\lim_kD_k$. In the ascending case the initial
structure on $C_k$ is generated by finite maxima of the structures of the
$D_n$ with $n\geq k$, among which that of $D_k$ occurs and dominates the
others, so $C_k=D_k$; the rest is dual.
\end{proof}

\noindent In both cases $\Lambda_D$ is the identity, so (C1) holds and with it
(C2). What degenerates is not the distinction between the two conditions but
the construction: both limits are the ordinary direct or inverse limit of the
family, and there is nothing to compare.

\begin{corollary}\label{cor:eventually}
If $D$ agrees with an ascending or a descending scale on some tail $N_{k_0}$,
then $\Lambda_D$ is an isomorphism.
\end{corollary}

\begin{proof}
Since $(C_k)$ is ascending in $k$ one has
$\lowlim D=\bigvee_kC_k=\bigvee_{k\geq k_0}C_k$, and each $C_k$ with
$k\geq k_0$ involves only the $D_n$ with $n\geq k_0$; dually
$\uplim D=\bigwedge_kE_k=\bigwedge_{k\geq k_0}E_k$ because $(E_k)$ is
descending. Both limits are therefore computed from $D|_{N_{k_0}}$, to which
Theorem~\ref{thm:degenerate} applies. Only the fact that $N_{k_0}$ is a tail is used.
\end{proof}

\begin{proposition}\label{prop:mutual}
If $D$ is both ascending and descending, then all $D_n$ are isomorphic
through the identity of their common underlying object and $\Lambda_D$ is an
isomorphism.
\end{proposition}

\begin{proof}
Mutual embedding gives $UD_n=UD_m$ for all $n,m$, and both induced identities
are morphisms of $\Cat$, so the structures are pairwise equivalent; in $\LCS$
they are equal. The
initial structure on $C_k$ is then equivalent to that of $D_k$ and the final
structure on $E_k$ likewise, and all transition morphisms of $(C_k)$ and
$(E_k)$ are isomorphisms.
\end{proof}

\begin{remark}\label{rem:mutual}
Proposition~\ref{prop:mutual} shows where counterexamples cannot be sought.
A family whose members are pairwise comparable in both directions converges,
however badly the comparison constants degenerate along the family, because
the initial and final structures are generated by \emph{finite} subfamilies.
To break \textnormal{(C1)} the scale must fail $\sqsubseteq$ in one direction
for infinitely many pairs, as the scale of Theorem~\ref{thm:separate} does.
\end{remark}

\begin{example}\label{ex:nonmono}
Convergence occurs off the ascending and descending families. Let
$V=\mathbb{K}^{(\mathbb{N})}$ with basis $(e_j)_{j\geq1}$, let
$\Cat=\Cat_0=\Vect$, and put
$D_n=\operatorname{span}\{e_1,\dots,e_n\}+\operatorname{span}\{e_{2n}\}$. This
scale is neither ascending nor descending: $e_4\in D_2\setminus D_3$ and
$e_3\in D_3\setminus D_2$.
For $k\geq2$ there is $n$ with $k<n<2k$, and $e_{2k}\notin D_n$ for such $n$,
so $C_k=\operatorname{span}\{e_1,\dots,e_k\}$; for $k=1$ no such $n$ exists
and $C_1=\operatorname{span}\{e_1,e_2\}$. In either case
$\lowlim D=\bigvee_kC_k=V$; every $e_j$ lies
in $D_n$ for $n\geq j$, so $E_k=V$ and $\uplim D=V$. Hence $\Lambda_D$ is the
identity.
\end{example}

\subsection{Where the two conditions coincide}

By Proposition~\ref{prop:reflect} the question is whether $U$ reflects
isomorphisms.
\[
\begin{array}{lll}
\Cat & \Cat_0 & \textnormal{(C2)}\Rightarrow\textnormal{(C1)}\\[3pt]
\text{subsets of a set, }\overline{\mathbb{R}}\text{, divisibility}
  & \text{itself} & \text{yes, trivially}\\
\Set,\ \Vect & \text{itself} & \text{yes, trivially}\\
\text{finite-dimensional objects of }\LCS & \Vect & \text{yes, one Hausdorff topology}\\
\Ban & \Vect & \text{yes, by the open mapping theorem}\\
\LCS & \Vect & \text{no, by Theorem~\ref{thm:separate}}
\end{array}
\]
In the last line the failure is not universal, and
Proposition~\ref{prop:ultrab} says where it cannot occur.

\section{The comparison as an interchange of quantifiers}\label{sec:quant}

The examples in this section are two-parameter: a family $F(n,m)$ of
subobjects of a fixed object, decreasing in $n$ and increasing in $m$, with
the two iterated limits $\colim_m\lim_nF$ and $\lim_n\colim_mF$ and the
canonical comparison between them. Membership in the first means that
\emph{some} $m$ serves \emph{every} $n$, and in the second that \emph{every}
$n$ is served by \emph{some} $m$. The index category may well be filtered, as
$\mathbb{N}^{\mathrm{op}}\times\mathbb{N}$ is; what prevents agreement is that
the limit is infinite, filtered colimits commuting only with finite limits.
The four instances below are classical objects in which this interchange seems
to us the right way to see the two constructions being compared. The conditions (C1) and (C2) are applied to these two-parameter families by
analogy; the statements below concern the two iterated limits directly.

\subsection{Test functions}

For a compact $K\subseteq\mathbb{R}$ let $C^m(K)$ be the Banach space of $C^m$
functions supported in $K$ with the norm $\max_{j\leq m}\sup|f^{(j)}|$, and
$\mathcal{D}^m=C^m_c(\mathbb{R})$ with its strict $(LB)$-topology. Fix compacta $K_1\subseteq K_2\subseteq\cdots$ exhausting $\mathbb{R}$ with
$K_\mu\subset\operatorname{int}K_{\mu+1}$, and for positive $\varepsilon_\mu$
put
\[
  W=\{f\in\mathcal{D}(\mathbb{R}):\ |f^{(j)}(x)|\leq\varepsilon_\mu
  \ \text{ for }x\notin\operatorname{int}K_{\mu-1},\ j\leq\mu,\ \mu\geq1\},
  \qquad K_0=\emptyset .
\]
This $W$ is absolutely convex, and it is a neighbourhood of $0$ for the
inductive topology because its trace on $\mathcal{D}_{K_\nu}$ is cut out by
finitely many conditions: for $\mu\geq\nu+2$ one has
$K_\nu\subset\operatorname{int}K_{\mu-1}$, so the condition of index $\mu$ is
vacuous there.

\begin{proposition}\label{prop:testfns}
Let $F(m,\mu)=C^m(K_\mu)$, contravariant in $m$ and covariant in $\mu$. Then
$\colim_\mu\lim_mF(m,\mu)=\mathcal{D}(\mathbb{R})$ and
$\lim_m\colim_\mu F(m,\mu)=\lim_m\mathcal{D}^m$; the comparison morphism is
the identity of $C_c^\infty(\mathbb{R})$, a continuous bijection which is not
open.
\end{proposition}

\begin{proof}
As sets $\bigcap_m C^m_c(\mathbb{R})=C^\infty_c(\mathbb{R})$, the support of a
function not depending on $m$, and the two identifications of the iterated
limits are immediate. The lower limit carries the inductive topology of the $\mathcal{D}_{K_\mu}$
and the upper limit the projective topology of the $\mathcal{D}^m$; the
comparison is the identity of $C^\infty_c(\mathbb{R})$ and is continuous
because each inclusion $\mathcal{D}(\mathbb{R})\to\mathcal{D}^m$ is. It
remains to show that it is not open, that is, that $W$ contains no
neighbourhood of $0$ for the projective topology. A base of neighbourhoods of $0$ for the projective
topology is given by the traces of neighbourhoods of $0$ in a single
$\mathcal{D}^m$, and such a neighbourhood contains
$\{f\in C^m(K_\mu):\|f\|_{C^m}<\delta_\mu\}$ for suitable $\delta_\mu>0$: the
order of differentiation it constrains is bounded by $m$. Fix $m$ and such a
neighbourhood, with constants $\delta_\mu$, and choose $\mu>m$; fix a non-zero $\psi\in\mathcal{D}_{K_\mu}$ whose support meets the interior
of $K_\mu\setminus K_{\mu-1}$, where the constraint of order $\mu$ in $W$ is
active; such a $\psi$ exists because $K_{\mu-1}\subset\operatorname{int}K_\mu$.
Set $f_N(x)=\lambda_N\psi(x)\cos(Nx)$ with
$\lambda_N=\delta_\mu\bigl(2^{m+1}N^m\|\psi\|_{C^m}\bigr)^{-1}$. By the Leibniz
rule $\|\psi\cos(N\cdot)\|_{C^m}\leq2^mN^m\|\psi\|_{C^m}$ for $N\geq1$, hence
$\|f_N\|_{C^m}\leq\delta_\mu/2<\delta_\mu$ and $f_N$ lies in the given
neighbourhood, while the derivative of order $\mu$
of $f_N$ has supremum at least $c\lambda_N N^{\mu}=c'N^{\mu-m}\delta_\mu$ on
$K_\mu\setminus K_{\mu-1}$, which tends to infinity. Hence $f_N\notin W$ for
large $N$.
\end{proof}

Here the index controlling size is the compactum: a test function lies in
the lower limit because some $K_\mu$ contains its support for all orders of
differentiation at once, and the $(LF)$-topology of $\mathcal{D}(\mathbb{R})$
may let the order grow with $\mu$, which no single $\mathcal{D}^m$ allows.

\subsection{Weighted Sobolev spaces}

For $k\in\mathbb{N}$ and $\delta\in\mathbb{R}$ let
\[
  H^{k,\delta}(\mathbb{R})=\big\{u\in L^2_{\mathrm{loc}}:\
  \langle x\rangle^{\delta}u^{(j)}\in L^2(\mathbb{R})\ \text{for }j\leq k\big\},
  \qquad \langle x\rangle=(1+x^2)^{1/2}.
\]
These are the spaces of Agmon's limiting absorption principle, in which
regularity is gained at the cost of decay.

\begin{proposition}\label{prop:sobolev}
Put $F(k,\delta)=H^{k,-\delta}(\mathbb{R})$ for $k,\delta>0$, contravariant in
$k$ and covariant in $\delta$. Then
\[
  \colim_\delta\lim_kF(k,\delta)\ \subsetneq\ \lim_k\colim_\delta F(k,\delta),
\]
so the family fails \textnormal{(C2)}.
\end{proposition}

\begin{proof}
Write $B=\colim_\delta\lim_kF$ and $A=\lim_k\colim_\delta F$. A distribution
lies in $B$ when a single $\delta$ serves every $k$, and in $A$ when for each
$k$ some $\delta=\delta(k)$ does; hence $B\subseteq A$.

Let $u(x)=\sin(x^2)$. By induction
$u^{(j)}(x)=p_j(x)\sin(x^2)+q_j(x)\cos(x^2)$ with $p_j,q_j$ real polynomials
of degree at most $j$, and differentiating the leading term shows that exactly
one of $p_j,q_j$ has degree $j$, with leading coefficient $\pm2^{j}$. Hence
$|u^{(j)}(x)|\leq C_j\langle x\rangle^{j}$, so $u\in F(k,\delta)$ as soon as
$\delta>k+\tfrac12$, and therefore $u\in A$.

For the lower bound, suppose the leading term is $p_j$, the other case being
symmetric. On the set $E=\{x>1:|\sin(x^2)|\geq\tfrac12\}$ we have
\[
  |u^{(j)}(x)|\ \geq\ \tfrac12\cdot2^{j}|x|^{j}-C'_j|x|^{j-1}\ \geq\ c_j|x|^{j}
\]
for $x\in E$ large, with $c_j>0$. The set $E$ has positive lower density. Indeed, the
substitution $\theta=x^{2}$ carries $[R,2R]$ onto $[R^{2},4R^{2}]$ with
$d\theta=2x\,dx\leq4R\,dx$, and $\{\theta:|\sin\theta|\geq\tfrac12\}$ occupies
a fraction $\tfrac23$ of each period of length $2\pi$, so
\[
  |E\cap[R,2R]|\ \geq\ \frac{1}{4R}\,
  \bigl|\{\theta\in[R^{2},4R^{2}]:|\sin\theta|\geq\tfrac12\}\bigr|
  \ \geq\ \frac{1}{4R}\cdot\frac{2}{3}\cdot3R^{2}(1-o(1))\ \geq\ \frac{R}{3}
\]
for large $R$. Consequently, summing over the intervals $[2^{i},2^{i+1}]$,
\[
  \int_E\langle x\rangle^{-2\delta}|u^{(j)}(x)|^{2}\,dx
  \ \gtrsim\ \sum_i 2^{i}\,2^{i(2j-2\delta)},
\]
which diverges unless $2\delta-2j>1$. Thus
$u\in F(j,\delta)$ forces $\delta>j+\tfrac12$, no single $\delta$ serves all
$j$, and $u\notin B$.
\end{proof}

\subsection{Laurent series and the adeles}

\begin{proposition}\label{prop:tate}
For $n,m\geq0$ let $F(n,m)=t^{-n}\Bbbk[[t]]/t^{m}\Bbbk[[t]]$, a finite-dimensional
$\Bbbk$-vector space, covariant in $n$ and contravariant in $m$, and compute all
limits in $\Vect_\Bbbk$. Then
\[
  \colim_n\lim_mF(n,m)=\Bbbk((t))=\lim_m\colim_nF(n,m),
\]
and the comparison morphism is the identity.
\end{proposition}

\begin{proof}
For fixed $n$ the transition maps in $m$ are the quotient maps, and for fixed
$m$ the transition maps in $n$ are injective, so
\[
  \lim_m F(n,m)=t^{-n}\Bbbk[[t]],
  \qquad
  \colim_n F(n,m)=\Bbbk((t))/t^{m}\Bbbk[[t]].
\]
Taking the colimit of the first over $n$ gives $\Bbbk((t))$. The other order gives
the limit of the second over $m$. The canonical map from $\Bbbk((t))$ into this limit is
injective because $\bigcap_m t^{m}\Bbbk[[t]]=0$.

It is surjective. A compatible family $(x_m)$ records coefficients $a_j$ for
all $j$, and each $x_m$ is the class of a Laurent series, so the set
$\{j<m:a_j\neq0\}$ is bounded below. Since $x_{m}$ is the restriction of
$x_{m+1}$, the same bound serves for every $m$, and $(a_j)_{j\in\mathbb{Z}}$
is a Laurent series.
\end{proof}

\begin{remark}\label{rem:tatecaveat}
Unlike the other examples of this section, $F$ is not a family of subobjects of
a fixed object: its members are subquotients, and its variances are opposite to
the convention fixed above. What is compared is the canonical morphism of
\cite[IX.2]{MacLane} for a two-parameter diagram, which here happens to be
invertible.
\end{remark}

\begin{remark}\label{rem:formal}
It would be wrong to read Proposition~\ref{prop:tate} as a statement about the
formal objects. Write $\mathcal{E}=\mathbf{Vect}^{f}_\Bbbk$ for the exact category of
finite-dimensional $\Bbbk$-vector spaces, which is idempotent complete, as the
result quoted below requires. Elementary Tate objects are usually
placed in $\mathrm{Ind}^a\mathrm{Pro}^a(\mathcal{E})$ \cite{BGW}, but they may
equally be realised in $\mathrm{Pro}^a\mathrm{Ind}^a(\mathcal{E})$: there is a
canonical fully exact embedding of the elementary Tate objects into the
latter, whose image consists of the admissible Pro-Ind objects admitting a
lattice \cite[Prop.~3.1]{BGHW}.
What Proposition~\ref{prop:tate} compares is not these formal objects but the
two iterated limits computed in $\Vect_\Bbbk$.
\end{remark}

For a single Laurent series field the two limits agree. For all of them at
once they do not.

\begin{example}\label{ex:adeles}
Work in the category of topological abelian groups. Let $P$ be the discrete
set of finite places of $\mathbb{Q}$, let $N$ be the directed set of finite
sets of places, and put $F(v,S)=\mathbb{Q}_v$ for
$v\in S$ and $F(v,S)=\mathbb{Z}_v$ otherwise, inside $V=\prod_{v\in P}\mathbb{Q}_v$.
Membership in the two iterated limits is an interchange of quantifiers:
\begin{align*}
  x\in\colim_S\prod_{v\in P}F(v,S)&\iff\exists S\ \forall v:\ x_v\in F(v,S),\\
  x\in\prod_{v\in P}\colim_SF(v,S)&\iff\forall v\ \exists S:\ x_v\in F(v,S).
\end{align*}
The first condition says that $x_v\in\mathbb{Z}_v$ for all but finitely many
$v$, so the lower limit is the ring of finite adeles $\mathbb{A}_f$; the
second holds for every $x$, since $S=\{v\}$ serves, so the upper limit is
$\prod_{v\in P}\mathbb{Q}_v$. The comparison morphism is the inclusion
$\mathbb{A}_f\hookrightarrow\prod_{v\in P}\mathbb{Q}_v$, which is neither surjective nor
open, the restricted product topology being strictly finer than the
topology induced on $\mathbb{A}_f$ by the product.
\end{example}

\begin{remark}\label{rem:adelesright}
One might expect the failure of the two limits to agree to be a defect of the
family. Here it is not. The lower limit
$\mathbb{A}_f$ is locally compact and carries a Haar measure, and it is on
this object that the arithmetic of $\mathbb{Q}$ is done; the upper limit, an
infinite product of non-compact locally compact groups, is not locally
compact, so the two limits are not isomorphic as topological groups. They are
of course isomorphic as abstract $\mathbb{Q}$-vector spaces, both having
dimension $2^{\aleph_0}$; it is the topology that distinguishes them.
The construction therefore separates the correct object from the naive one,
and places the correct one below. The same holds for the space of test
functions, where by Proposition~\ref{prop:testfns} the lower limit is
$\mathcal{D}(\mathbb{R})$ with its usual $(LF)$-topology. In both cases the
lower limit is defined by an existential quantifier over the index that
controls size, $\exists S$ here and $\exists\mu$ there, which is the general form of the lower limit for a two-parameter family.
\end{remark}

\section{Other definitions one might give}\label{sec:other}

Definition~\ref{def:grades} offers two conditions, and each of them is
asserted somewhere in the literature. Three further conditions suggest themselves. Two of them we reject, and the
third we do not pursue here; each is a genuine idea with a genuine source,
and it is worth saying precisely what happens to it in the present setting.

Proposition~\ref{prop:factor} indicates in advance where the trouble lies. A
condition strictly between (C2) and (C1) is obtained by forgetting part of the
structure and asking that $\Lambda_D$ become an isomorphism there, and the
condition is a genuine one exactly when structure is lost at that stage. Of
the three conditions below, the first forgets the morphism $\Lambda_D$ itself
and retains only its source and target; the second leaves $\Lambda_D$ alone
and changes the indexing; only the third comes from a genuine enrichment of
the members of the scale, which is why it is the only one we do not reject.

\subsection{Convergence up to isomorphism}

The first is to ask only that $\lowlim D$ and $\uplim D$ be isomorphic. The
motivation is that many questions asked about a limit are invariant under
isomorphism: whether it is an $(LF)$-space, barrelled, reflexive, nuclear.
For such questions the isomorphism type is all that is needed, and a statement
of the form ``the limit is isomorphic to $X$'' is what one often has.

Remark~\ref{rem:isoweak} shows what happens if the condition is adopted. In a
category whose objects are isomorphic to proper subobjects of themselves it is
satisfied by scales whose limits differ by an arbitrary codimension, and it is
satisfied by the alternating scale of $\ell^2$ and a hyperplane, where the
comparison morphism is a proper inclusion. It is moreover not functorial: it
specifies no morphism, so nothing can be transported along it. It is a
property of the pair of limits, not of the comparison, and we use it only in
that role.

\subsection{Convergence after reindexing}

The second is to ask that $\Lambda_{D|_{N'}}$ be an isomorphism for some
cofinal $N'\subseteq N$. The motivation is the oldest device in analysis: a
sequence that does not converge may have a subsequence that does, and one
argues with the subsequence. The source is genuine, but the device does not
survive the transfer, for a reason specific to scales: a scale is a family and
not a diagram, so restricting to a cofinal subset can change both limits.

\begin{example}\label{ex:subfamily}
Let $\Cat$ be the poset of subsets of $\mathbb{N}$ and let $A_n=\{0\}$ for $n$
even, $A_n=\{n\}$ for $n$ odd. Then $\bigcap_{n\geq k}A_n=\emptyset$ for every
$k$, so $\liminf_nA_n=\emptyset$, while an odd $m$ is absent from
$\bigcup_{n\geq k}A_n$ once $k>m$, so $\limsup_nA_n=\{0\}$; the scale does
not converge. Along the cofinal subset $2\mathbb{N}$ it is constant, and so
converges.
\end{example}

\begin{proposition}\label{prop:subfamily}
Let $N'\subseteq N$ be cofinal. If $N'$ contains a tail of $N$, then
$\Lambda_{D|_{N'}}$ is an isomorphism if and only if $\Lambda_D$ is. If the
complement of $N'$ is also cofinal, then for any subobjects
$X\sqsubseteq Y$ of $V$ with $X\neq Y$ the scale equal to $X$ on $N'$ and to
$Y$ off $N'$ has $\Lambda_{D|_{N'}}$ an isomorphism and
$\lowlim D=X\neq Y=\uplim D$.
\end{proposition}

\begin{proof}
If $N'\supseteq N_{k_0}$, then $N'\cap N_k=N_k$ for $k\geq k_0$, so the meets
and joins over tails agree with $C_k$ and $E_k$, and since $(C_k)$ ascends and
$(E_k)$ descends the two limits are unchanged. If both $N'$ and its complement
are cofinal, then $D|_{N'}$ is constant, hence convergent by
Theorem~\ref{thm:degenerate}, while every tail meets both, giving $C_k=X$ and
$E_k=Y$.
\end{proof}

A cofinal subset either contains a tail or has cofinal complement, so the
condition is either equivalent to (C1) or compatible with an arbitrary
prescribed discrepancy between the limits. It is worth adding that this is not
the Mittag-Leffler condition, with which it shares only the word ``tail'':
Mittag-Leffler constrains the images of the transition morphisms of an inverse
system, whereas the present condition places no condition on transition
morphisms at all.

The correct form of the subsequence device is the Urysohn principle, in which
every subfamily is required to have a further subfamily converging to one and
the same limit. For subsets of a set this is a characterisation of
convergence.

\begin{proposition}\label{prop:urysohn}
Let $(A_n)_{n\in\mathbb{N}}$ be subsets of a set $X$ and $A\subseteq X$. Then
$\liminf_nA_n=\limsup_nA_n=A$ if and only if every infinite
$N'\subseteq\mathbb{N}$ contains an infinite $N''$ with
$\liminf_{n\in N''}A_n=\limsup_{n\in N''}A_n=A$.
\end{proposition}

\begin{proof}
If $A_n\to A$ then every infinite $N''$ has the same limits, as passing to a
subset preserves lying in $A_n$ for all large $n$ and lying in $A_n$ only
finitely often. Conversely, let $x\in\limsup_nA_n$; then $\{n:x\in A_n\}$ is
infinite, and on the $N''$ it contains $x\in A_n$ throughout, so
$x\in\liminf_{N''}A_n=A$. If $x\in A\setminus\liminf_nA_n$, then
$\{n:x\notin A_n\}$ is infinite, and on the $N''$ it contains $x$ lies in no
$A_n$, so $x\notin\limsup_{N''}A_n=A$, a contradiction. Hence
$\limsup_nA_n\subseteq A\subseteq\liminf_nA_n$.
\end{proof}

The common limit must be understood as a subobject and not merely up to
isomorphism: in Remark~\ref{rem:isoweak} the even and odd subfamilies converge
to $\ell^2$ and to $H$, which are isomorphic but distinct. Even so, the
principle does not survive the passage from sets to spaces.

\begin{proposition}\label{prop:urysohnfails}
The Urysohn form fails for subspaces of a vector space. Let $V$ have basis
$e_0,e_1,e_2,\dots$ and put $D_{2n}=\operatorname{span}(e_0+e_n)$ and
$D_{2n+1}=\operatorname{span}(e_n)$ for $n\geq1$. Every infinite subfamily of
$D$ contains an infinite subfamily converging to $0$, while $\lowlim D=0$ and
$e_0\in\uplim D$. Conversely, a convergent scale converges along every
infinite subfamily, to the same limit.
\end{proposition}

\begin{proof}
An infinite set of indices contains infinitely many of one parity. Along
indices of one parity the members are pairwise distinct lines, so every meet
over a tail is $0$; and a vector in the span of a tail of the even members,
or of a tail of the odd members, has its support outside $\{0\}$ contained in
that tail, so the meet of these spans over all tails is $0$. Hence both
limits along such a subfamily are $0$. In the whole family every tail contains
$e_0+e_n$ and $e_n$ for some $n$, so $e_0=(e_0+e_n)-e_n$ lies in every $E_k$
and in $\uplim D$, while $\lowlim D=0$ because distinct lines meet in $0$.
For the converse, passing to a subfamily can only enlarge the meets over
tails and shrink the joins, so
$\lowlim D\subseteq\lowlim D|_{N''}\subseteq\uplim D|_{N''}\subseteq\uplim D$,
and if the outer two agree so do the inner ones.
\end{proof}

The proof for sets uses that a join of subsets is their union, so that
membership in the upper limit is witnessed by a single index. In the lattice
of subspaces the join is a span, and the witness may require several indices
at once; this is exactly what the example exploits.

\subsection{Convergence in two topologies}

The third is to compute the two members of the comparison in different
topologies. This is Mosco convergence \cite{Mosco}, in which the lower limit
is taken in the strong topology and the upper in the weak one, the requirement
on a family $(A_n)$ of sets being
$\mathrm{w}\text{-}\limsup A_n\subseteq A\subseteq\mathrm{s}\text{-}\liminf A_n$;
for quadratic forms it is equivalent to strong resolvent convergence of the
associated operators, and it is the form in which variational convergence is
usually applied. It requires each member of the
scale to carry two comparable topologies, which our examples do not; we note
it as the natural variant for scales that do, and do not pursue it.

\section{Open questions}\label{sec:problems}

\begin{question}\label{q:between}
Corollary~\ref{prop:born} answers in the affirmative the question whether some
condition lies strictly between the two, and Proposition~\ref{prop:factor}
says where such conditions come from. Which factorisations of the forgetful
functor of $\LCS$ over $\Vect$ give conditions distinct from these two, and is
there one between \textnormal{(C2)} and bornological convergence? The two-topology
condition of \S\ref{sec:other}, for scales whose members carry a weak and a
strong topology as in Mosco convergence \cite{Mosco}, is the candidate we
would look at first.
\end{question}

\begin{question}\label{q:urysohnclass}
For which ambient objects does the Urysohn principle of
Proposition~\ref{prop:urysohn} hold? It holds when the subobject lattice is a
lattice of subsets, and fails for subspaces of a vector space by
Proposition~\ref{prop:urysohnfails}. What the proof uses is not distributivity but the
concrete form of the operations: that the meet of subobjects is their
intersection and the join their union, so that membership in the upper limit
is witnessed by a single index. For which categories does that hold, and is it
necessary?
\end{question}

\begin{question}\label{q:ultrabchar}
For a scale of Banach spaces indexed by $\mathbb{N}$, when is $\uplim D$
ultrabornological? By Proposition~\ref{prop:ultrab} this suffices for \textnormal{(C2)} to imply
\textnormal{(C1)}, and for weighted
scales the answer is contained in the analysis of \cite{ABB}.
\end{question}

\section*{FUNDING}

This work was not supported by any funding.

\section*{CONFLICT OF INTEREST}

The authors declare that they have no conflicts of interest.

\end{document}